\documentclass[twoside,a4,12pt]{amsart}

\usepackage[english]{babel}
\usepackage[utf8x]{inputenc}
\usepackage{amsmath}
\usepackage{graphicx}
\usepackage[colorinlistoftodos]{todonotes}
\usepackage{amsmath,amscd,amsthm}
\usepackage{amstext, amsfonts, a4}
\usepackage{amssymb}
\usepackage{tikz-cd}
\usepackage{fancyhdr}
\usepackage{enumerate}
\usepackage{cleveref}
\usepackage{xcolor}

\swapnumbers
\numberwithin{equation}{section}
\newtheorem{thm}[equation]{Theorem}

\newtheorem{prop}[equation]{Proposition}
\newtheorem{cor}[equation]{Corollary}

\newtheorem{lem}[equation]{Lemma}

\theoremstyle{definition}

\newtheorem{qu}[equation]{Question}

\renewcommand{\dim}{\operatorname{\mathsf{dim}}}
\renewcommand{\deg}{\operatorname{\mathsf{deg}}}
\newcommand\ind{\operatorname{\mathsf{ind}}}
\renewcommand\exp{\operatorname{\mathsf{exp}}}

\newcommand\corr{\operatorname{\mathsf{cor}}}

\newcommand{\wt}{\widetilde}
\newcommand{\la}{\langle}
\newcommand{\ra}{\rangle}
\newcommand{\mg}[1]{{#1}^{\times}}
\newcommand{\sq}[1]{{#1}^{\times 2}}
\newcommand{\n}{\mathsf{N}}
\newcommand{\ca}[2]{{}^{#1}\!\!{#2}}

\newcommand{\car}{\mathsf{char}}
\newcommand{\zz}{\mathbb{Z}}
\newcommand{\nat}{\mathbb{N}}

\renewcommand{\leq}{\leqslant}
\renewcommand{\geq}{\geqslant}
\renewcommand{\max}{\mathsf{max}}

\begin{document}
	\title{Splitting quaternion algebras defined over a finite field extension}

\date{04.10.2019}
	
\author{Karim Johannes Becher}
\author{Fatma Kader B\.{i}ng\"{o}l}
\author{David B.~Leep}

\address{Universiteit Antwerpen, Departement Wiskunde, Middelheim\-laan~1, 2020 Ant\-werpen, Belgium.}

\email{KarimJohannes.Becher@uantwerpen.be}
\email{FatmaKader.Bingol@uantwerpen.be}
\address{Department of Mathematics, University of Kentucky, Lexington, KY 40506-0027, USA.}
\email{leep@uky.edu}
\thanks{This work was supported by the FWO Odysseus programme (project \emph{Explicit Methods in Quadratic Form Theory}), funded by the Fonds Wetenschappelijk Onderzoek -- Vlaanderen.}
	\maketitle
	
\begin{abstract}
We study systems of quadratic forms over fields and their isotropy over $2$-extensions.
We apply this to obtain particular splitting fields for quaternion algebras defined over a finite field extension.
As a consequence we obtain that every central simple algebra of degree $16$ is split by a $2$-extension of degree at most $2^{16}$.

\medskip\noindent
{\sc Classification} (MSC 2010): 12E15, 16K20 
 
	\end{abstract}

\section{introduction}

Any central division algebra is split by a finite Galois extension. In particular, any central simple algebra is
 Brauer equivalent to a crossed product.
Examples of central division algebras which are not crossed products have been constructed in all degrees which are not of the form $d$ or $2d$ for a positive square-free integer $d$ (see \cite{Ami72}).
On the other hand, no such example is known where the exponent is square-free.

In this article we focus on central simple algebras of exponent $2$ and exploit their relation to quadratic forms.
By Merkurjev's Theorem, any such algebra is Brauer equivalent to a tensor product of quaternion algebras over its center.
Generic arguments show that there is a universal bound for the number of quaternion algebras needed to obtain such a product, hence~a bound in terms of the index of the algebra which is independent of its center (see~{\cite{Tig84}}).

The problem of obtaining such a bound is related to the problem of finding a bound on the degree of a Galois extension that splits a given Brauer equivalence class.
As a first step towards a solution to both problems we aim to bound the degree of a $2$-extension of the center which is a splitting field of a central simple algebra of exponent $2$ and of a given index.
%
%


Our approach is mainly based on the study of systems of quadratic forms.
In Section 2 we show that a system of $r$ quadratic forms in more than $\frac{r(r+1)}2$ variables over a field has a non-trivial zero over some $2$-extension of degree at most $2^r$ (\Cref{T:sys-bound}).
In Section 3 we consider quaternion algebras defined over a finite extension of a field $F$ and study splitting fields which are obtained as field composita with $2$-extensions of $F$. 
For a finite field extension $K/F$ with $[K:F]\leq8$ we show that any quaternion algebra over $K$ has a splitting field of the form $KF'$ for some $2$-extension $F'/F$ with $[F':F]\leq2^{[K:F]}$ (\Cref{C:quat-over-deg8ext}).
In Section~4 we obtain an obstruction for such a quaternion algebra over $K$ to be split by a $2$-extension coming from $F$ which is of small degree. 
This obstruction is given in terms of the corestriction algebra.
Our argument is a variation of an argument from \cite{RT83}.
In Section 5, we apply the result from Section 3 to show that any central division algebra of degree $16$ over $F$ is split by a $2$-extension of $F$ of degree at most $2^{16}$ (\Cref{T:2split16}).

\section{System of quadratic forms}

Let $F$ be a field.
Let $r$ be a positive integer and let $V$ be a finite-dimensional vector space over $F$. We consider a system of $r$ quadratic forms $\varphi=(\varphi_1,\ldots,\varphi_r)$ on $V$. The system $\varphi$ is said to be \emph{isotropic} if there exists a nonzero vector $v\in V$ such that $\varphi_i(v)=0$ for $1\leq i\leq r$, and \emph{anisotropic} otherwise. 

For a field extension $K/F$, $\varphi$ extends naturally to a system of quadratic forms over $K$ defined on the $K$-vector space $K\otimes_FV$, which is denoted by  $\varphi_K$.


A separable finite field extension $K/F$ is called a $2$-\emph{extension} if the degree of its normal closure is a power of $2$. Since $2$-groups are solvable, it follows by Galois theory that the finite extension $K/F$ is a $2$-extension if and only if $K$ can be obtained from $F$ by a sequence of separable quadratic extensions. 

We call a finite field extension $K/F$ a \emph{generalized $2$-extension} if there exist $r\in\nat$ and intermediate fields $K_0,\dots,K_r$ of $K/F$ with $K_0=F$, $K_r=K$ and such that $K_i/K_{i-1}$ is a quadratic field extension for $1\leq i\leq r$.
In characteristic different from $2$, all quadratic field extensions are separable, and consequently generalized $2$-extensions are in fact $2$-extensions in the usual sense.

\begin{thm}\label{T:sys-bound}
	Let $r\geq 1$ and let $\varphi$ 
	be a system of $r$ quadratic forms defined on an $F$-vector space of dimension at least $\frac{r(r+1)}{2}+1$. There exists a generalized $2$-extension $F'/F$ with $[F':F]\leq2^r$ such that $\varphi_{F'}$ is isotropic. 
\end{thm}
\begin{proof}
Every anisotropic quadratic form of dimension at least two is isotropic over some quadratic field extension. Hence the statement is true for $r=1$.
We prove the statement by induction on $r$.

For the induction step we assume that $r\geq 2$.
Let $\varphi=(\varphi_1,\ldots,\varphi_r)$ and let $V$ denote the $F$-vector space on which the system $\varphi$ is defined.
We fix a nonzero vector $v\in V$. If $\varphi_{i}(v)=0$ for $1\leq i\leq r$, then $\varphi$ is isotropic over $F$ and we are done. 
Otherwise we may assume that $\varphi_r(v)\neq0$. Set $a_i=\varphi_i(v)$ for $1\leq i\leq r$. Consider the system of $r$ quadratic forms $\wt{\varphi}=(\wt{\varphi}_1,\ldots,\wt{\varphi}_{r-1},\varphi_r)$ on $V$ where $\wt{\varphi}_i=a_r\varphi_i-a_i\varphi_r$ for $1\leq i< r$. Note that $\wt{\varphi}_i(v)=0$ for $1\leq i< r$. Since over any field extension $F'/F$ the systems $\varphi_{F'}$ and $\wt{\varphi}_{F'}$ have the same isotropic vectors, the claim is not affected when we replace $\varphi$ by $\wt{\varphi}$. Hence we may assume that $\varphi_i(v)=0$ for $1\leq i< r$. 
	
For $1\leq i<r$, we set 
$$W_i=\{x\in V\mid \varphi_i(x+v)=\varphi_i(x)\}\,;$$ 
as $\varphi_i(v)=0$, $W_i$ is 
 the orthogonal space of $Fv$ with respect to the symmetric bilinear form associated to $\varphi_i$, and in particular $W_i$ is an $F$-subspace of $V$ of codimension at most $1$. 
We set $W=\bigcap_{i=1}^{r-1}W_i$.
By the hypothesis we have $\dim_FV\geq\frac{r(r+1)}{2}+1$ 
 and therefore
  $$\dim_F W\geq \dim_F V-(r-1)\geq\hbox{$\frac{(r-1)r}{2}$}+2\,.$$
Note that $v\in W$. We fix a direct complement $V'$ of $Fv$ in $W$.
The restriction of $\varphi_{1},\ldots,\varphi_{r-1}$ to $V'$ defines a system of $(r-1)$-quadratic forms on $V'$, and we have that $\dim_FV'=\dim_FW-1\geq \hbox{$\frac{(r-1)r}{2}$}+1$.
Hence, the induction hypothesis yields that there exist a generalized $2$-extension $F''/F$ with $[F'':F]\leq2^{r-1}$ and a nonzero vector $w\in V'\otimes_F F''$ such that $(\varphi_{i})_{F''}(w)=0$ for $1\leq i<r$. 
We consider the restriction of the quadratic form $(\varphi_{r})_{F''}$ to  $F''v\oplus F''w$. 
By the observation at the beginning of the proof, there exists a field extension $F'/F''$ with $[F':F'']\leq 2$ and a nonzero vector $u\in F'v\oplus F'w$ such that $(\varphi_{r})_{F'}(u)=0$. 
The choices of $v$ and $w$ imply that the forms $(\varphi_{1})_{F'},\ldots,( \varphi_{r-1})_{F'}$ vanish on $F'v\oplus F'w$.
In particular, we obtain that $$(\varphi_{1})_{F'}(u)=\ldots=(\varphi_{r-1})_{F'}(u)=(\varphi_{r})_{F'}(u)=0\,.$$ 
Furthermore, $F'/F$ is a generalized $2$-extension with $[F':F]\leq2^{r-1}\cdot2=2^r$.
\end{proof}

We call a field \emph{quadratically closed} if it has no quadratic field extension.
As a consequence of \Cref{T:sys-bound} we retrieve a result from \cite[Theorem 2.6]{Leep84}:

\begin{cor}\label{C:system-bound}
	If $F$ is quadratically closed then any system of $r$ quadratic forms in strictly more than $\frac{r(r+1)}{2}$ variables is isotropic.
\end{cor}
\begin{proof}
	Let $\varphi$ be a system of $r$ quadratic forms on an $F$-vector space $V$ with $\dim_FV\geq \frac{r(r+1)}{2}+1$. Then, by \Cref{T:sys-bound}, there exists a generalized $2$-extension $F'/F$ with $[F':F]\leq2^r$  such that $\varphi_{F'}$ is isotropic. However, since $F$ is quadratically closed, we get that $F'=F$. Therefore $\varphi$ is isotropic.
\end{proof}

In \cite[Chapter~9, 2.3~(1)]{Pfi95} examples are given which show that, for $r\leq 3$ and $F$ the quadratic closure of $\mathbb{Q}$, the lower bounds on the number of variables in terms of $r$ are best possible to obtain the conclusions in \Cref{T:sys-bound} and \Cref{C:system-bound}. 






We turn briefly to a statement which will be useful to cover the case of characteristic two in \Cref{P:form-split-2ext-basefield} below.
In \cite[Lemma 7.6]{Alb39} a statement similar to the following one is proven, but without giving an explicit bound on the degree of the inseparable extension.

\begin{lem}\label{L:Albert}
Let $p$ be a prime number and  assume that $\car F=p$. Let $K/F$ be an algebraic field extension and let $\alpha\in K^{\times}$. 
Let $r\in\nat$ be minimal such that $\alpha^{p^r}$ is separable over $F$. There exists a purely inseparable extension $L/F$ with $L^{p^{r+1}}\subseteq F$ and $[L:F]\leq p^{[F(\alpha^{p^r}):F]	\cdot(r+1)}\leq p^{[K:F]}$ such that $\alpha\in {L(\alpha)}^{p}$.
\end{lem}
\begin{proof}
The hypothesis on $r$ implies that $[F(\alpha):F(\alpha^{p^r})]=p^r$.
We set $n=[F(\alpha^{p^r}):F]$. Then $[F(\alpha):F]=n\cdot p^r$.
Let $g(X)$ denote the minimal polynomial of $\alpha^{p^r}$ over $F$.  
We can choose a purely inseparable extension $L/F$ and elements $c_1,\dots,c_n\in L$ such that 
$$g(X)=X^n+c_{1}^{p^{r+1}}X^{n-1}+\ldots+c_{n-1}^{p^{r+1}}X+c_n^{p^{r+1}}$$
in $L[X]$ and $L=F({c}_1,\ldots,{c}_{n})$.
Note that $[L:F]\leq p^{n(r+1)}\leq p^{np^r}$ and $L^{p^{r+1}}\subseteq F$.  

Let $$g'(X)=X^n+{c}_{1}X^{n-1}+\ldots+{c}_{n-1}X+{c}_n\in L[X].$$
Then $g'(X)^{p^{r+1}}=g(X^{p^{r+1}})$, whereby $g'(\sqrt[p]{\alpha})^{p^{r+1}}=g(\alpha^{p^r})=0$.
Since $L/F$ is purely inseparable and $\alpha^{p^r}$ is separable over $F$, we obtain that $$[L(\alpha^{p^r}):L]=[F(\alpha^{p^r}):F]=n=\deg(g')\geq [L(\sqrt[p]{\alpha}):L]\,,$$ whereby 
$L(\alpha^{p^r})=L(\alpha)=L(\sqrt[p]{\alpha})$.
Hence $\alpha\in {L(\alpha)}^{p}$.
\end{proof}

%

For $n\in\nat$ and $a_1,\dots,a_n\in \mg{F}$ we denote by $\la a_1,\dots,a_n\ra$ the $n$-dimensional diagonal quadratic form 
$F^n\to F, (x_1,\dots,x_n)\mapsto \sum_{i=1}^n a_ix_i^2$.

\begin{lem}\label{L:quadratic-slot}
	Let  $K/F$ be a finite field extension. Let  $\alpha\in K^{\times}$ and $g\in F[X]$ with $g(\alpha)\neq 0$ and $\deg(g)\leq 2$. Consider the quadratic form $\varphi=\langle1,\alpha,g(\alpha)\rangle$ over $K$. There exists a generalized $2$-extension $F'/F$ with $[F':F]\leq2^3$ such that $\varphi_{KF'}$ is isotropic.
\end{lem}
\begin{proof}
If $g\in {F}^\times$ then $\varphi$ is isotropic over $KF'$ for $F'=F(\sqrt{-g})$.
Assume next that $g(X)=a(X+b)$ for some $a,b\in F$. Then we set $F'=F(\sqrt{-a},\sqrt{b})$. Note that $-g(X)=(\sqrt{-a})^2(X+(\sqrt{b})^2)$.
Hence $-g(\alpha)$ is represented by the quadratic form $\langle1,\alpha\rangle$ over $F'(\alpha)$.
Therefore $\varphi_{KF'}$ is isotropic.

Assume finally that $g(X)=a(X^2+bX+c)$ with $a,b,c\in F$. In this case we set $F'=F(\sqrt{-a},\sqrt{c},\sqrt{b-2\sqrt{c}})$. Then $-g(X)=(\sqrt{-a})^2((X+\sqrt{c})^2+(b-2\sqrt{c})X)$.
Hence $-g(\alpha)$ is represented by the quadratic form $\langle1, \alpha\rangle$  over $F'(\alpha)$.
Therefore $\varphi_{KF'}$ is isotropic.

In each case we obtain that $\varphi_{KF'}$ is isotropic and that $F'/F$ is a generalized $2$-extension with $[F':F]\leq 2^{\deg(g)+1}\leq 2^3$.
\end{proof}

Let $K/F$ be a finite field extension and $s: K\to F$ a nonzero $F$-linear functional. Let $\varphi : V\to K$ be a quadratic form over $K$. 
We consider $V$ as an $F$-vector space and obtain a quadratic form  $s\circ\varphi:V\to F$ over $F$.

In the proof of the next result we use the standard concepts of 
isometry $\simeq$ and orthogonal sum $\perp$ for quadratic forms.

\begin{prop}\label{P:form-split-2ext-basefield}
Let $K/F$ be a finite field extension and $r=[K:F]$. Let $\varphi$ be a quadratic form over $K$ defined on a $K$-vector space $V$ with $\dim_KV\geq 
\max(3,\frac{r}{2})$. 
There exists a {generalized} $2$-extension $F'/F$ with $[F':F]\leq2^{r}$ and such that $\varphi_{KF'}$ is isotropic.
\end{prop}
\begin{proof}
We may assume that $\varphi$ is anisotropic, as otherwise we may take $F'=F$. In particular $\varphi$ is regular. 
We may further scale $\varphi$ without affecting the claim, hence we may assume that $\varphi$ represents $1$. 
If $\car(F)=2$, then as $\dim_KV\geq3$, it follows by \Cref{L:Albert} that there exists a purely inseparable field extension $F'/F$ with $[F':F]\leq 2^r$ such that $\varphi_{KF'}$ is isotropic, and $F'/F$ is clearly a generalized $2$-extension.
 
 Assume now that $\car(F)\neq2$.
We first prove the statement under the additional assumption that $K$ does not contain any quadratic extension of $F$. Note that this implies that $K/F$ is linearly disjoint from any $2$-extension of $F$.	
	

	 
We first assume that $r\geq7$. Set $n=\dim_KV$. Then $n\geq \frac{r}2 >3$. We write $\varphi\simeq\langle1,\alpha\rangle\perp\varphi'$ with $\alpha\in K^{\times}$ and a quadratic form $\varphi'$ defined on a $K$-vector space $V'$ with $\dim_KV'=n-2$. 
If $(1,\alpha,\alpha^2)$ is $F$-linearly dependent, then $[F(\alpha):F]\leq 2$ and hence $F'=F(\sqrt{-\alpha})$ is a $2$-extension of $F$ with $[F':F]\leq 4$, and  $\varphi$ becomes isotropic over $KF'$. Assume now that $(1,\alpha,\alpha^2)$ is $F$-linearly independent. Then there exists an $F$-basis $(\alpha_0,\ldots,\alpha_{r-1})$ of $K$ with $\alpha_i=\alpha^i$ for $i=0,1,2$. 
For $0\leq i<r$, let $s_i : K\to F$ be the $F$-linear functional given by $\alpha_j\mapsto \delta_{ij}$ for $0\leq j<r$ (where $\delta_{ij}$ denotes the Kronecker delta).  
We view $V'$ as an $F$-vector space and obtain a system  of quadratic forms $\phi=(s_{3}\circ\varphi',\ldots,s_{r-1}\circ\varphi')$ on $V'$. 
Since $\phi$ is a system of $r-3$ quadratic forms over $F$ and since
    $$\dim_F V'=(n-2)\cdot r>\left(\mbox{$\frac{r-1}{2}$}+\mbox{$\frac{3}{r}$}\right)\cdot r-2r=\mbox{$\frac{1}2$}(r-3)(r-2),$$ 
\Cref{T:sys-bound} implies that there exists a $2$-extension $F'/F$ with $[F':F]\leq2^{r-3}$ and such that $\phi_{F'}$ is isotropic. Let $v\in V'\otimes_F F'$ be a nonzero isotropic vector of $\phi_{F'}$. 
Since $F'/F$ is linearly disjoint from $K/F$, $(\alpha_0,\ldots,\alpha_{r-1})$ is an $F'$-basis of $KF'$ and we get that $s_i(\varphi'_{KF'}(v))=0$ for $3\leq i<r$. 
We conclude that $\varphi'_{KF'}(v)\in F'\oplus F'\alpha\oplus F'\alpha^2$.
Hence $\varphi'_{KF'}(v)=g(\alpha)$ for some $g\in F'[X]$ with $\deg(g)\leq 2$. If $g(\alpha)=0$, then $\varphi'$ is isotropic over $KF'$ and we are done. Assume now that $g(\alpha)\neq0$. 
Since the quadratic form $\varphi'$ is regular and $\car(F)\neq 2$, it follows that $\varphi'_{KF'}\simeq\langle g(\alpha)\rangle\perp \varphi''$ for some quadratic form $\varphi''$ over $KF'$.
Thus $$\varphi_{KF'}\simeq\langle1,\alpha,g(\alpha)\rangle\perp\varphi''$$ over $KF'$.
By \Cref{L:quadratic-slot} there exists a $2$-extension $F''/F'$ with $[F'':F']\leq2^3$ and such that the form 
$\langle1,\alpha,g(\alpha)\rangle$ is isotropic over $KF''$.
Then the quadratic form $\varphi_{KF''}$ is isotropic, and $F''/F$ is a $2$-extension with $$[F'':F]=[F'':F']\cdot[F':F]\leq2^3\cdot2^{r-3}=2^r\,.$$
This shows the statement in the case where $r\geq 7$.
	
We now consider the case where $r\leq 6$. Note that $\dim_KV\geq3$.
We may assume without loss of generality that $\dim_KV=3$. 
Since $\car(F)\neq 2$, we can write $\varphi\simeq\langle1,\alpha,\beta\rangle$ for certain $\alpha,\beta\in K^{\times}$. 
We consider the $2$-fold Pfister form   $\pi=\langle1,\alpha,\beta,\alpha\beta\rangle$  over $K$. 
Let $(\alpha_0,\ldots,\alpha_{r-1})$ be an $F$-basis of $K$. For $0\leq i<r$, let $s_i : K\to F$ be the $F$-linear functional given by $\alpha_j\mapsto \delta_{ij}$ for $0\leq j<r$.
We consider the system $\phi=(s_{0}\circ\pi,\ldots,s_{r-1}\circ\pi)$ defined on the $F$-vector space $K^4$. 
Hence $\phi$ is a system of $r$ quadratic forms in $4\cdot r$ variables over $F$. 
As $r\leq6$ we have that $4\cdot r>\frac{r(r+1)}{2}$.
Hence, by \Cref{T:sys-bound}, there exists a $2$-extension $F'/F$ with $[F':F]\leq2^r$ such that $\phi$ becomes isotropic over $F'$. 
Since $F'/F$ is linearly disjoint from $K/F$,
it follows that $\pi_{KF'}$ is isotropic. 
Since $\pi$ is a $2$-fold Pfister form, it follows that $\pi_{KF'}$ is hyperbolic, and as $\varphi$ is a $3$-dimensional subform of $\pi$ we conclude that $\varphi_{KF'}$ is isotropic.

Hence we have proven the statement under the additional assumption that $K$ does not contain any quadratic extension of $F$.
In the general case we choose a maximal $2$-extension $K_0/F$ contained in $K/F$.
Then $K$ does not contain any quadratic extension of $K_0$.
By applying the previous argument to $K/K_0$ we obtain that
there exists a $2$-extension $F'/K_0$ such that $\varphi_{KF'}$ is isotropic and $[F':K_0]\leq2^{[K:K_0]}$. Then $F'/F$ is a $2$-extension, and furthermore 
$$[F':F]=[F':K_0]\cdot[K_0:F]\leq 2^{[K:K_0]}\cdot [K_0:F]\leq 2^{[K:F]}= 2^r\,.\vspace{-7mm}$$
\end{proof} 

For the special case of a quadratically closed base field, we obtain a slight improvement to the  bound for the growth of the $u$-invariant under finite field extensions obtained in \cite[Theorem 2.10]{Leep84}.

\begin{cor}
\label{C:quadclose system-bound}
	Let $F$ be quadratically closed and let $K/F$ be a finite field extension with $[K:F]=r$. Then every quadratic form  over $K$ of dimension greater than or equal to $\max(3,\frac{r}{2})$ is isotropic.
\end{cor}
\begin{proof}
By the hypothesis $F$ has no proper generalized $2$-extension.
Hence the statement follows directly from \Cref{P:form-split-2ext-basefield}.
\end{proof}

\section{Quaternion algebras defined over a finite extension} 

The following sections are an addition to the classical theory of central simple algebras over a field as it is covered in many textbooks, for example in \cite{Alb39}, \cite{Draxl} or \cite{GS06}.
Given a central simple algebra $A$, we use the fundamental relations between the degree, index and exponent of $A$,  which we denote by $\deg A$, $\ind A$ and $\exp A$, respectively.

A \emph{quaternion algebra} over a field is a central simple algebra of degree two.
A quaternion algebra is \emph{split} if it is isomorphic to a $(2\times 2)$-matrix algebra.
Any quaternion algebra is either split or a division algebra.

Let $F$ be a field. We will use
 a characteristic-independent description of quaternion algebras by generators and relations (see \cite[Chapter~IX, Section 10]{Alb39}):
For any $a\in F$ with $-4a\neq1$ and $b\in F^{\times}$, the $4$-dimensional $F$-vector space with basis $(1,i,j,ij)$ is endowed with an $F$-algebra multiplication given by the relations
$$i^2-i=a,\quad j^2=b \quad \text{and}\quad ji=(1-i)j;$$
this construction yields an $F$-quaternion algebra, which is denoted by $[a,b)_F$. 
One easily checks that every $F$-quaternion algebra is isomorphic to $[a,b)_F$ for certain $a\in F$ and $b\in F^{\times}$.

Let $Q$ be an $F$-quaternion algebra. 
The reduced norm map $\n_Q:Q\to F$ is a quadratic $2$-fold Pfister form over $F$. 
Moreover, $\n_Q$ is isotropic if and only if $Q$ is split (see \cite[Corollary 12.5]{EKM}).

For a central simple $F$-algebra $A$ and a field extension $K/F$, we have that $A\otimes_FK$ is a central simple $K$-algebra, which we denote by $A_{K}$.
%

\begin{qu}\label{Q}
Let $K/F$ be a finite field extension and let $Q$ be a $K$-quaternion algebra. Does there exist a $2$-extension $F'/F$ such that $Q_{KF'}$ is split?
\end{qu}

When the characteristic is different from $2$, we obtain the following positive answer to this question in small degrees for the field extension.

\begin{thm}\label{C:quat-over-deg8ext}
Assume that $\car (F)\neq 2$.
Let $K/F$ be a finite field extension with $[K:F]\leq 8$ 
and let $Q$ be a $K$-quaternion algebra. 
Then there exists a $2$-extension $F'/F$ with $[F':F]\leq2^{[K:F]}$ and such that $Q_{KF'}$ is split.
\end{thm}
\begin{proof}
Set $r=[K:F]$. Since $r\leq8$, we have that $\dim_K Q=4\geq\max(3,\frac{r}{2})$. Hence, by \Cref{P:form-split-2ext-basefield}, $\n_Q$ becomes isotropic over $KF'$ for some  $2$-extension $F'/F$ with $[F':F]\leq2^r$. 
It follows that $Q_{KF'}$ is split.
%
\end{proof}

If $\car(F)\neq 2$, we do not know how to answer \Cref{Q} for quaternion algebras defined over a finite field extension $K/F$ with $[K:F]>8$.
In characteristic $2$, we may consider a weaker version of \Cref{Q}, where 
 $F'/F$ is allowed to be a generalized $2$-extension  (which may be purely inseparable) instead of a $2$-extension.
The following statement gives a positive answer to this variation of \Cref{Q}.

\begin{prop}\label{C:multiquad-split-quaternion-char2}
Assume that $\car (F)=2$.
Let $K/F$ be a finite field extension
and let $Q$ be a $K$-quaternion algebra. 
Then there exists a purely inseparable field extension $F'/F$ with $[F':F]\leq 2^{[K:F]}$ such that $Q_{KF'}$ is split. Moreover, $F'/F$ can be chosen as a multiquadratic extension when $K/F$ is separable.
\end{prop}
\begin{proof}	
We have $Q\simeq[\alpha,\beta)_K$ for certain $\alpha\in K$ and $\beta\in K^{\times}$.
By \Cref{L:Albert}, there exists a purely inseparable field extension $F'/F$ such that $\beta\in {F'(\beta)}^{\times 2}$ and $[F':F]\leq 2^{[K:F]}$. Then $Q_{KF'}$ is split.
When $K/F$ is separable, it  follows further by \Cref{L:Albert}  that $F'^2\subseteq F$, whereby $F'/F$ is multiquadratic.
\end{proof}

\section{The corestriction of a quaternion algebra}

In this section we provide examples showing that the bounds in \Cref{C:quat-over-deg8ext} and \Cref{C:multiquad-split-quaternion-char2} on the degree of the extension $F'/F$ are optimal when the degree of $K/F$ is odd. 
First we recall the definition of the corestriction of a central simple algebra and  some of its properties. For this purpose we follow \cite[Part I, Section 8]{Draxl}.

Let $K/F$ be a finite Galois extension and let $G$ denote its Galois group. Let $A$ be a central simple $K$-algebra. For $\sigma\in G$ we define the conjugate algebra $\ca{\sigma}{A}$ as follows: the underlying ring structure of $\ca{\sigma}{A}$ is the same as of $A$ and the scalar multiplication from $K$ is given by $\lambda\cdot a=\sigma(\lambda)a$ for any $\lambda\in K$ and $a\in A$. Note that $\ca{\sigma}{A}$ is a central simple $K$-algebra with $\deg \ca{\sigma}{A}=\deg A$. 

Now assume that $K/F$ is a cyclic extension and let $\sigma$ be a generator of $G$. Set $r=[K:F]$. Consider the central simple $K$-algebra $$A^{\otimes G}=A\otimes\ca{\sigma}{A}\otimes\cdots\otimes\ca{\sigma^{r-1}}{A}\,,$$ where the tensor product is taken over $K$. For $0\leq i< r$, the $F$-automorphism $\sigma^i$ of $K$ induces an $F$-algebra automorphism on $A^{\otimes G}$ such that 
$$\sigma^i\cdot (a_0\otimes a_1\otimes\cdots\otimes a_{r-1})=a_{r-i}\otimes\cdots\otimes a_{r-1}\otimes a_0\otimes \cdots\otimes a_{r-1-i}$$ 
for any $(a_0,\dots, a_{r-1})\in\prod_{i=0}^{r-1}\,\ca{\sigma^{i}}{A}$. 
This determines a group action of $G$ on the $K$-algebra $A^{\otimes G}$. The \emph{corestriction of} $A$, denoted $\corr_{K/F} A$, is defined to be the $F$-subalgebra of $A^{\otimes G}$ consisting of the elements that are fixed by this action of~$G$. Clearly, $\corr_{K/F}A$ does not depend, up to $F$-isomorphism, on the choice of $\sigma$. 

\begin{prop}\label{P:properties of cor}
	Let $K/F$ be a cyclic Galois extension. Let $A$ and $B$ be central simple $K$-algebras. The following hold:
	\begin{enumerate}[$(a)$]
	    \item $(\corr_{K/F}A)\otimes_F K\simeq A^{\otimes G}$.
	    \item $\corr_{K/F} A$ is a central simple $F$-algebra and $\deg(\corr_{K/F}A)=(\deg A)^{[K:F]}$.
	    \item If $A$ is split then $\corr_{K/F}A$ is split.
	    \item $\corr_{K/F}(A\otimes_K B)\simeq_F(\corr_{K/F}A)\otimes_F(\corr_{K/F}B)$.
	\end{enumerate}
\end{prop}
\begin{proof}
See \cite[Section 8]{Draxl}. 
\end{proof}

\begin{cor}\label{L:commutativity of cor-res}
	Let $K/F$ be a cyclic Galois extension. Let $L/F$ be a field extension which is linearly disjoint from $K/F$. Let $A$ be a central simple $K$-algebra. We have that $(\corr_{K/F}A)\otimes_F L\simeq\corr_{KL/L}(A_{KL})$.
\end{cor}
\begin{proof}
Note that $KL/L$ is a Galois extension.
We denote by $G$ and $\Gamma$ the Galois groups of the extensions $K/F$ and $KL/L$, respectively.
Since $L/F$ is linearly disjoint from $K/F$, the map
$\Gamma\to G,\tau\to \tau|_K$ is an isomorphism.
We can view $A^{\otimes G}$ as a $K$-subalgebra of  $(A_{KL})^{\otimes \Gamma}$. 
In this way, $\corr_{K/F}A$ becomes an $F$-subalgebra of $\corr_{KL/L}(A_{KL})$.
Since these are central simple algebras over $F$ and over $L$, respectively, we 
obtain a natural embedding of central simple $L$-algebras $\corr_{K/F}A\otimes_F L\to\corr_{KL/L}(A_{KL})$.
Since $$\deg (\corr_{K/F}A\otimes_F L)=(\deg A)^{[K:F]}=(\deg A_{KL})^{[KL:L]}=\deg(\corr_{KL/L}(A_{KL})),$$ we obtain that this embedding is an isomorphism.
%
%
\end{proof}

\begin{cor}\label{C:splitting field of cor}
	Let $K/F$ be a cyclic Galois extension and $A$ a central simple $K$-algebra. Let $L/F$ be a finite field extension linearly disjoint from $K/F$ and such that $A_{KL}$ is split. Then $\corr_{K/F} A\otimes_F L$ is split.
\end{cor}
\begin{proof}
We have by \Cref{L:commutativity of cor-res} that $\corr_{K/F} A\otimes_F L\simeq \corr_{KL/L}(A_{KL})$, and by \Cref{P:properties of cor} this $L$-algebra is split.
\end{proof}

We will consider the corestriction of the generic quaternion algebra defined over a cyclic field extension of given degree.

Let $k$ be a field and $r$ be a positive integer. 
We set $$K=k(X_0,\ldots,X_{r-1},Y_0,\ldots,Y_{r-1})\,,$$
where $X_0,Y_0,\dots,X_{r-1},Y_{r-1}$ are variables over $k$. 
Let $\sigma$ be the $k$-linear automorphism of $K$ given by $\sigma(X_i)=X_{i+1}$ and $\sigma(Y_i)=Y_{i+1}$ for $0\leq i< r-1$, $\sigma(X_{r-1})=X_0$ and $\sigma(Y_{r-1})=Y_0$. Let $G$ be the group generated by $\sigma$. Then $G$ is a cyclic group of order $r$. Let $F=K^G$, the fixed field of $G$. It follows that $K/F$ is a cyclic extension with Galois group $G$ and with $[K:F]=|G|=r$. 

The following is a variation of \cite[Proposition]{RT83}.

\begin{prop}\label{P:sharpness-bound-2ext-quat-over deg8ext}
Let $Q=[X_0,Y_0)_K$. Then $\corr_{K/F} Q$ is a central $F$-division algebra of degree $2^r$ and exponent $2$. In particular, for any finite extension $L/F$ linearly disjoint from $K/F$ and such that $Q_{KL}$ is split, we have that $[L:F]\geq2^r$.
\end{prop}
\begin{proof}
By \Cref{P:properties of cor} we have that 
$\corr_{K/F}Q$ is a central simple $F$-algebra with $\deg(\corr_{K/F}Q)=2^r$, because $r=[K:F]$. Note that 
$$\ca{\sigma^i\,\,}{[X_0,Y_0)_K}\simeq_K[\sigma^i(X_0),\sigma^i(Y_0))_K=[X_i,Y_i)_K\,\,\mbox{ for }\,\,0\leq i< r\,.$$ 
Hence $$Q^{\otimes G}=Q\otimes\ca{\sigma\,\,}{Q}\otimes\cdots\otimes\ca{\sigma^{r-1}\,}{Q}\simeq[X_0,Y_0)_K\otimes_K \cdots\otimes_K [X_{r-1},Y_{r-1})_K,$$
and this is a central division algebra over $K$. 
It follows from the item $(a)$ of \Cref{P:properties of cor} that $\corr_{K/F}Q$ is a central $F$-division algebra. 

Consider now a finite field extension $L/F$ which is linearly disjoint from $K/F$ and such that $Q_{KL}$ is split. Then $\corr_{K/F}Q\otimes_FL$ is split, by Corollary \ref{C:splitting field of cor}. 
Since $\corr_{K/F}Q$ is a division algebra, it follows that $[L:F]\geq \deg(\corr_{K/F}Q)=2^r$.
Moreover, by \Cref{P:properties of cor}, we obtain that $(\corr_{K/F}Q)\otimes_F(\corr_{K/F}Q)$ is split. Hence $\corr_{K/F}Q$ has exponent~$2$. 
\end{proof}

\Cref{P:sharpness-bound-2ext-quat-over deg8ext} shows that the bound in \Cref{C:quat-over-deg8ext} is optimal for extensions of degree $3$, $5$ and $7$, and that the bound in \Cref{C:multiquad-split-quaternion-char2} is optimal for extensions of odd degree.

\section{Splitting fields for algebras of exponent $2$} 

In this final section we use \Cref{C:quat-over-deg8ext} to obtain splitting fields for central simple algebras of degree $16$ which are $2$-extensions of bounded degree of the center (\Cref{T:2split16}).
In the characteristic two case we obtain a more general result for central simple algebras of exponent two and arbitrary degree (\Cref{T:2ext-split-exp2csa-char2}). 

We recall the following fact, which is essentially due to Albert and Rowen.

\begin{prop}\label{P:Rowen}
Let $A$ be a central simple $F$-algebra with $\ind A=2^r$ where $r\leq 3$.
There exists a $2$-extension $K/F$ with $[K:F]\leq2^{5}$ such that $A_K$ is split.
Moreover, if $r\leq 2$ or $\exp A=2$, then there exists a Galois extension $K/F$ with Galois group $(\zz/2\zz)^r$  such that $A_{K}$ is split.
\end{prop}
\begin{proof}
Let $D$ be the central $F$-division algebra which is Brauer equivalent to $A$.
Note that $A$ and $D$ have the same splitting fields over $F$.

If $r=1$, then $D$ is an $F$-quaternion algebra and the statement is obvious.
For $r=2$ the statement is due to Albert, see \cite[Theorem 11.9]{Alb39}.

Assume now that $r=3$.
Consider first the case where $\exp A=2$.
By Rowen's Theorem from \cite{Row78} and \cite{Row84},
$D$ contains a maximal subfield $K$ such that $K/F$ is a Galois extension with Galois group $(\zz/2\zz)^3$; see \cite[Corollary 7.7]{BGBT18a} for another proof of this result.
Then $A_K$ is split and $K/F$ is a $2$-extension with $[K:F]=8$.

In the more general situation we have that $\exp A=2^{s}$ with $1\leq s\leq 3$.
It follows by \cite[Lemma 5.7]{Alb39} that $\ind (A\otimes_F A)\leq 2^{r-1}=4$.
From the previous cases we obtain that there exists a $2$-extension $F'/F$ such that $[F':F]\leq 4$ and such that $(A\otimes_F A)_{F'}$ is split.
Then $\exp A_{F'}\leq 2$ and $\ind A_{F'}$ divides $8$.
It follows from the previous cases that $A_{K}$ is split for a Galois $2$-extension $K/F'$ with $[K:F']\leq 8$. 
Then $K/F$ is a $2$-extension and $[K:F]=[K:F']\cdot [F':F]\leq 8\cdot 4=2^5$.
\end{proof}


The following three statements are focussing on fields of characteristic $2$.
The following lemma  bounds the degree of a $2$-extension splitting a central simple algebra by the degree of a generalized $2$-extension with the same property. 

\begin{lem}\label{L:char2-reduc-gen2ext}
Let $A$ be a central simple $F$-algebra and let $K/F$ be a generalized $2$-extension such that $A_K$ is split.
Then there exists a $2$-extension $L/F$ with $[L:F]\leq [K:F]$ and such that $A_L$ is split.
\end{lem}
\begin{proof}
To prove the statement we may  replace $F$ by the separable closure of $F$ in $K$.
Hence we may assume that $K/F$ is purely inseparable.
We prove the claim by induction on $[K:F]$.
If $[K:F]=1$, then $K=F$ and we may take $L=F$.
Assume now that $[K:F]\geq 2$.
Then there exists an intermediate field $K'$ of $K/F$ such that $[K:K']=2$. If $A_{K'}$ is split then the induction hypothesis applied to $K'/F$ yields a $2$-extension $L/F$ with $[L:F]\leq[K':F]=\frac{1}{2}[K:F]$ such that $A_L$ is split and we are done.

Assume that $A_{K'}$ is not split. Hence $\ind A_{K'}=2$. Then there exists a separable quadratic extension $L'/K'$ such that $A_{L'}$ is split. Since $K'/F$ is a purely inseparable extension we have that $L'=K'F'$ for some separable quadratic extension $F'/F$. As $[L':F']=[K':F]=\frac{1}{2}[K:F]$, it follows by applying the induction hypothesis to the central simple $F'$-algebra $A_{F'}$ that there exists some $2$-extension $L/F'$ with $[L:F']\leq[L':F']$ such that $A_L$ is split. Now $L/F$ is a $2$-extension such that $A_L$ is split and $[L:F]=[L:F']\cdot[F':F]\leq2\cdot[L':F']=[K:F]$.
\end{proof}


\begin{lem}\label{L:insepindred}
Assume that $\car (F)=2$. Let $A$ be a central simple $F$-algebra with $\exp A=2$. Let $n\in\mathbb{N}$ be such that $\ind A=2^n$. Then there exists a purely inseparable extension $F'/F$ with $F'^2\subseteq F$, $[F':F]\leq2^{2^{n-1}} $ and $\ind A_{F'}\leq 2^{n-1}$.
\end{lem}
\begin{proof}
We consider towers of field extensions $M/L/F$
with $[M:F]=2^n$ such that $M/L$ is purely inseparable and $A_M$ is split.
We choose such a tower with $[L:F]$ as small as possible.
Then $L/F$ is separable.
By \cite[Lemma 4]{Bec16}  there exists a field extension
$L'/F$ with $[L':F]=\frac{1}2\ind A$ and such that $\ind A_{L'}=2$.
Hence there exists an inseparable quadratic extension $M'/L'$ such that $A_{M'}$ is split. Then $[M':F]=2^n$ and the choice of $M/L/F$ yields that $[L:F]\leq [L':F]=2^{n-1}$.
We conclude that $[L:F]=2^m$ for some $m\in\nat$ with $m<n$.

Hence $[M:L]\geq 2$, and since $M/L$ is purely inseparable, there exists an element $c\in(\mg{L}\cap\sq{M})\setminus\sq{L}$.
Since $\ind A=[M:F]=2^n$, $L(\sqrt{c})\subseteq M$ and $[L(\sqrt{c}):F]=2^{m+1}$, it follows that $\ind A_{L(\sqrt{c})}\leq2^{n-m-1}$.
By \Cref{L:Albert},
there exists a purely inseparable field extension $F'/F$ with $[F':F]\leq2^{2^{n-1}}$ and $F'^2\subseteq F$ such that 
$c\in \sq{F'L}$. Hence $L(\sqrt{c})\subseteq F'L$.
It follows that $\ind A_{F'L}\leq 2^{n-m-1}$.
Since $[LF':F']=[L:F]=2^m$, we conclude that $\ind A_{F'}\leq 2^{n-1}$.
\end{proof}

For a central simple algebra $A$ of exponent $2$ over a field $F$ of characteristic~$2$, it is shown in \cite[Theorem 1.1]{Flo13} that there exists a purely inseparable field extension $K/F$ with $[K:F]\leq2^{\ind A-1}$ such that $A_K$ is split and $K^2\subseteq F$. 
By \cite[Theorem 7.28]{Alb39} this gives rise to a splitting field $L/F$ of $A$ which is a $(\mathbb{Z}/2\mathbb{Z})^n$-Galois extension for $n=\ind A-1$, so in particular a $2$-extension. 
We now obtain a slightly better bound for the degree of a $2$-extension splitting~$A$.

\begin{prop}\label{T:2ext-split-exp2csa-char2}
Assume that $\car (F)=2$. Let $A$ be a central simple $F$-algebra with $\exp A=2$ and $\ind A\geq 8$.  
Then there exist a purely inseparable extension $F'/F$ with  
$[F':F]\leq 2^{\ind A-8}$ and a $(\zz/2\zz)^3$-Galois extension $K/F$ such that $A_{F'K}$ is split.
Moreover, there exists a $2$-extension $L/F$ with $[L:F]\leq2^{\ind A-5}$ such that $A_{L}$ is split.
\end{prop}
\begin{proof}
We have $\ind A=2^n$ for some integer $n\geq 3$.
Using \Cref{L:insepindred} repeatedly, we obtain for $1\leq k\leq n-3$  
 a purely inseparable field extension $F_k/F$ with $[F_k:F]\leq2^{(\sum_{i=1}^k {2^{n-i}})}$, $(F_k)^{2^k}\subseteq F$ and $\ind A_{F_k}= 2^{n-k}$.
We set $F'=F_{n-3}$ and obtain that $[F':F]\leq 2^{\ind A-8}$ and $\ind A_{F'}= 8$.

Note that  $\exp A_{F'}=\exp A=2$.
It follows by \Cref{P:Rowen} that there exists a $(\zz/2\zz)^3$-Galois extension $K'/F'$ such that $A_{K'}$ is split.
Let $K$ be the separable closure of $F$ in $K'$. Then $F'K=K'$ and $K/F$ is a  $(\zz/2\zz)^3$-Galois extension.

Finally, since  $F'K/F$ is a generalized $2$-extension  such that $A_{F'K}$ is split, it follows by \Cref{L:char2-reduc-gen2ext} that there exists a $2$-extension $L/F$ such that $A_{L}$ is split and $[L:F]\leq [F'K:F]=[K:F]\cdot [F':F]\leq 2^{\ind A-5}$.
\end{proof}

For our final result we return to fields of arbitrary characteristic.

\begin{thm}\label{T:2split16}
Let $A$ be a central simple $F$-algebra with $\ind A=16$.
Then $A$ is split by a $2$-extension $K/F$ with $[K:F]\leq2^{16}$.
Moreover, if $\exp A= 2$, then $A$ is split by a $2$-extension $K/F$ with $[K:F]\leq2^{11}$. 
\end{thm}
\begin{proof}
We first consider the case where $\exp A= 2$. 
If $\car(F)=2$, then it follows by \Cref{T:2ext-split-exp2csa-char2} that $A_K$ is split for some $2$-extension $K/F$ with $[K:F]\leq 2^{11}$.
Suppose that $\car(F)\neq 2$.
It follows by \cite[Lemma 4]{Bec16} that there exists 
a field extension $L/F$  such that $[L:F]=8$ and $\ind A_L=2$. 
We obtain by \Cref{C:quat-over-deg8ext} that there exists a $2$-extension $F'/F$ with $[F':F]\leq2^8$ such that $A_{LF'}$ is split. 
Since $[L:F]=8$, it follows that $\ind A_{F'}$ divides $8$.
Hence, by \Cref{P:Rowen}, there exists a $2$-extension $K/F'$ with $[K:F']\leq 8$ and such that $A_K$ is split. 
Then $[K:F]=[K:F']\cdot [F':F]\leq2^{11}$.
This settles the case where $\exp A=2$.


Consider now the general case. By \cite[Lemma 5.7]{Alb39}, we have $\ind(A\otimes_F A)\leq 8$. 
By \Cref{P:Rowen}, there exists a $2$-extension $F'/F$ with $[F':F]\leq2^5$ and such that $(A\otimes_F A)_{F'}$ is split. 
Then $\exp A_{F'}\leq 2$, and it follows from the special case treated above that $A_K$ is split for some $2$-extension $K/F'$ with $[K:F']\leq 2^{11}$.
Then $K/F$ is a $2$-extension with $[K:F]\leq 2^{16}$.
\end{proof}

	
\end{document}